\documentclass[12pt]{amsart}
\usepackage{amsmath,amsthm,enumerate,graphics,amsfonts,latexsym,amsopn,verbatim,amscd,amssymb}
\usepackage{color}
\usepackage{tikz}
\usepackage{amsmath,amssymb,graphicx,tikz,amsthm}
\usetikzlibrary{arrows,matrix}
\usepackage{url}

\theoremstyle{plain}
\newtheorem{thm}{Theorem}[section]

\newtheorem{prop}[thm]{Proposition}
\newtheorem{conj}[thm]{Conjecture}

\newtheorem{problem}{Theorem}[section] 
\newtheorem{pprob}[problem]{Problem}
\newenvironment{prob}{\pprob\rm}{\endpprob}

\theoremstyle{definition}

\DeclareMathOperator{\spec}{Spec}
\DeclareMathOperator{\L-spec}{L-spec}
\DeclareMathOperator{\Q-spec}{Q-spec}

\DeclareMathOperator{\cent}{Cent}
\DeclareMathOperator{\cl}{Cl}
\DeclareMathOperator{\nil}{Nil}
\DeclareMathOperator{\sol}{Sol}

\DeclareMathOperator{\diam}{diam}

\DeclareMathOperator{\PSL}{\mathop{\mathrm{PSL}}}

\DeclareMathOperator{\SmallGroup}{SmallGroup}

\tikzset{every path/.style=thick,
	acteur/.style={
		circle,
		fill=red,
		thick,
		inner sep=1pt,
		minimum size=0.15cm
}}

\usepackage{ragged2e}

\begin{document}

\title[A survey on  conjugacy class graphs of groups]{A survey on   conjugacy class graphs of groups}

\author[P. J. Cameron, F. E. Jannat, R. K. Nath  and R. Sharafdini]{Peter J. Cameron, Firdous Ee Jannat, Rajat  Kanti Nath  and Reza Sharafdini}

\address{P. J. Cameron, School of Mathematics and Statistics, University of St Andrews, North Haugh, St Andrews, Fife, KY16 9SS, UK.}
\email{pjc20@st-andrews.ac.uk}

\address{Firdous Ee Jannat, Department of Mathematical Science, Tezpur  University, Napaam -784028, Sonitpur, Assam, India.}

\email{firdusej@gmail.com}
\address{Rajat  Kanti Nath, Department of Mathematical Science, Tezpur  University, Napaam -784028, Sonitpur, Assam, India.}

\email{rajatkantinath@yahoo.com}

\address{Reza Sharafdini, Department of Mathematics,  Persian Gulf University, Bushehr 75169-13817, Iran.}
\email{sharafdini@pgu.ac.ir}

\begin{abstract}
	There are several graphs defined on groups.
	Among them we consider graphs  whose vertex set consists conjugacy classes of a group $G$ and adjacency is defined by properties of the elements of  conjugacy classes. In particular, we consider commuting/nilpotent/solvable  conjugacy class graph  of $G$ where two distinct conjugacy classes $a^G$ and $b^G$ are adjacent if there exist some elements $x\in a^G$ and $y\in b^G$ such that $\langle x, y \rangle$ is abelian/nilpotent/solvable. After a section of introductory results and examples, we discuss all the available results on connectedness, graph realization, genus, various spectra and energies of certain induced subgraphs of these graphs.  Proofs of the results are not included. However, many open problems for further investigation  are stated.
\end{abstract}

\thanks{ }

\maketitle

\section{Introduction}
Characterizing finite groups using graphs defined over groups has gained traction as a research topic in recent times. A number of graphs have been defined on groups (see \cite{PC-21}), among which the commuting graph has been studied widely.
Let $G$ be a finite non-abelian group. The \emph{commuting graph} of $G$ is a simple undirected graph whose vertex set is $G$, in which two vertices $x$ and $y$ are adjacent if they commute. The complement of this graph is the \emph{non-commuting graph} of $G$. The concept of commuting graph appeared in an important work of  Brauer and Fowler \cite{BF-55}, in the year 1955, a step towards the Classification of Finite Simple Groups. After the work of Erd\H{o}s and Neumann \cite{B12} on its complement in the year 1976, it was studied in its own right.

The property that $x$ and $y$ commute is equivalent to saying that $\langle x, y \rangle$ is abelian. Using  other group types such as cyclic, nilpotent, solvable, \dots\ graphs have been defined on groups.
Given a  group type $\mathcal{P}$ (for instance, cyclic, abelian, nilpotent, solvable etc.),  we define a graph  on a group $G$, called the $\mathcal{P}$ graph of
$G$, whose vertex set is $G$ and  two distinct vertices $x$ and $y$ are adjacent if $\langle x, y \rangle$ is a $\mathcal{P}$ group. In this nomenclature `abelian graph' is nothing but the commuting graph. These graphs forms the following hierarchy (where $A \subseteq B$ denotes that $A$ is a spanning subgraph of $B$):
\begin{equation}\label{H-1}
	\text{ Cyclic graph $\subseteq$ Commuting graph $\subseteq$ Nilpotent graph $\subseteq$ Solvable graph.}
\end{equation}

It is worth mentioning that  there are other graphs in the above hierarchy (for details one can see \cite{PC-21}). 
 
A \emph{dominant vertex} of a graph is a vertex that is adjacent to all other vertices. Let $\mathcal{P}(G) = \{g \in G : \langle g, h \rangle \text{ is a $\mathcal{P}$ group for all } h \in G\}$. Then $\mathcal{P}(G)$ is  the set of all dominant vertices of $\mathcal{P}$ graph of $G$. 

In the four cases just described, $\mathcal{P}(G)$ is a subgroup of $G$, the cyclicizer, centre, hypercentre, and solvable radical of $G$ respectively. The question of connectedness of the subgraphs of $\mathcal{P}$ graph induced by $G \setminus \mathcal{P}(G)$ is an interesting problem (see \cite{IJ-2008,GP-2013,G-2013,MP-2013,BNN-2020,ALMM-2020,BLN-2023}). Of course, this problem is trivial unless
$\mathcal{P}(G)$ is removed (otherwise the graph has diameter at most~$2$).
However for other studies such as independence number or clique number, it
makes either no difference or just a trivial difference.
  
Graphs are also  defined from (finite) groups by considering the vertex set as the set of conjugacy classes (or class sizes), with adjacency defined by certain properties of the elements of  conjugacy classes or  the class sizes.  A survey on graphs whose vertex set consists of class sizes of a finite group can be found in \cite{L-2008}. Graphs whose vertex set consists conjugacy classes of a group and adjacency is defined by properties of their sizes were first considered in \cite{BHM-1990}. 

In this survey, we shall consider graphs whose vertex set consists conjugacy classes of a group $G$, with adjacency defined by properties of the elements of these classes. We call such a graph the \emph{$\mathcal{P}$ conjugacy class graph} of $G$, or for short the $\mathcal{P}$CC-graph. The  $\mathcal{P}$ conjugacy class graph of $G$ is a simple undirected graph whose vertex set is the set of all the conjugacy classes of $G$ and two vertices (conjugacy classes) $a^G$ and $b^G$ are adjacent if there exist some elements $x\in a^G$ and $y\in b^G$ such that $\langle x, y \rangle$ is a $\mathcal{P}$ group. We have the following hierarchy in case of $\mathcal{P}$ conjugacy class graph:
\begin{equation}\label{H-2}
\text{Cyclic CC-graph $\subseteq$ CCC-graph $\subseteq$ NCC-graph $\subseteq$
SCC-graph,}
\end{equation}
where here and subsequently we use CCC-graph, NCC-graph, SCC-graph to denote commuting, nilpotent and solvable conjugacy class graph. (Commuting conjugacy class  graph is synonymous with `abelian  conjugacy class  graph'.)
Clearly, $1^G$ (the conjugacy class of the identity element) is a dominant vertex if   $\mathcal{P}(G)$ is a subgroup.  To make the question of
connectedness interesting, we should consider the induced subgraph on the set of non-dominant vertices. However, as we will see, it is not always known what this set is. Sometimes we just remove the identity class from the vertex set.

Note that the cyclic conjugacy class graph of a group has not yet been studied.
In what follows, we shall consider commuting/nilpotent/solvable conjugacy class graph.

\medskip

An outline of the paper follows. In the next section we give some general
results about conjugacy class graphs, including a discussion of when they
are complete (Theorem~\ref{t:complete}), when it happens that the
$\mathcal{P}$ graph is a ``blow-up'' of the $\mathcal{P}CC$-graph 
(Theorem~\ref{t:super}), and some discussion of the dominant vertices (a
characterisation is known only for the CCC-graph, see
Proposition~\ref{DV-CCC-graph} and Problem~\ref{Prob-DV-set}). The following
three sections survey the three graph types, discussing connectedness, detailed
structure for special groups, and properties such as genus, spectrum and
energy. (These results are taken from the literature and proofs are not given.)
The final section includes some open problems.

\section{General remarks and examples}

We begin with a general observation about conjugacy class graphs which is
often useful. Let $\mathcal{P}$ be any group-theoretic property, and let
$\Gamma$ be the $\mathcal{P}CC$-graph of $G$; that is, the
vertices are conjugacy classes, and there is an edge $\{C_1,C_2\}$ if and
only if there exist $g_i\in C_i$ for $i=1,2$ such that the group
$\langle g_1,g_2\rangle$ has property $\mathcal{P}$. In fact a stronger
condition holds: if $\{C_1,C_2\}$ is an edge, then for any $h_1\in C_1$
there exists $h_2\in C_2$ such that $\langle h_1,h_2\rangle$ has property $\mathcal{P}$.
For let $g_1,g_2$ be as in the definition. There exists $x\in G$ such that
$g_1^x=h_1$; then, letting $h_2=g_2^x$, we see that
\[\langle g_1,g_2\rangle^x=\langle h_1,h_2\rangle,\]
and since $\mathcal{P}$ is a group-theoretic property it is preserved by conjugation.

The first result relevant to conjugacy class graphs is the theorem of Landau
\cite{landau} from 1903: given the number of conjugacy classes of a finite
group $G$, there is an upper bound on the order of $G$. This implies the
following result.

\begin{prop}
Given a graph $\Gamma$, there are only finitely many finite groups $G$ whose
commuting, nilpotent or solvable conjugacy class graph is isomorphic to
$\Gamma$.
\end{prop}

The number of such groups is not usually $1$. For example, for any abelian
group of order $n$, the CCC-graph is the complete graph on $n$ vertices.

\begin{prob}\label{gp-determination}
\rm
Which finite groups are uniquely determined by their commuting, nilpotent, or
solvable conjugacy class graph?
\end{prob}

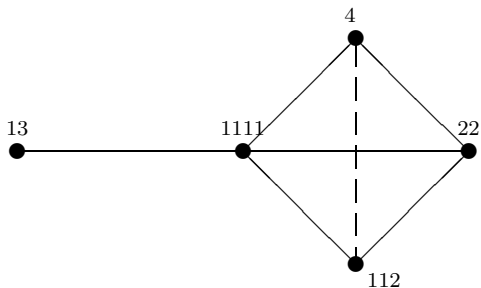
\begin{figure}[htbp]
\begin{center}
\setlength{\unitlength}{1.5mm}
\begin{picture}(40,20)
\multiput(0,10)(20,0){3}{\circle*{1.5}}
\multiput(30,0)(0,20){2}{\circle*{1.5}}
\put(0,10){\line(1,0){40}}
\multiput(30,0)(-10,10){2}{\line(1,1){10}}
\multiput(30,0)(10,10){2}{\line(-1,1){10}}
\multiput(30,-0.5)(0,3){7}{\line(0,1){2}}
\put(18,11.5){$\scriptstyle{1111}$}
\put(-1,11.5){$\scriptstyle{13}$}
\put(39,11.5){$\scriptstyle{22}$}
\put(29,21.5){$\scriptstyle{4}$}
\put(31,-2){$\scriptstyle{112}$}
\end{picture}
\end{center}
\caption{\label{f:s4}Conjugacy class graphs of $S_4$}
\end{figure}

Figure~\ref{f:s4} shows the conjugacy classes of the symmetric group $S_4$, with
the CCC- and NCC-graphs. (The conjugacy classes of $S_n$ are defined by cycle
types, and are labelled with partitions of $n$; so their number is $p(n)$,
where $p$ is the partition function, sequence A000041 in the On-line
Encyclopedia of Integer Sequences~\cite{oeis}.) Solid
lines show edges in the CCC-graph, while the dotted line is the additional
edge in the NCC-graph. Note that the partitions $1111$, $112$, $22$ and $4$
form a clique in the NCC-graph. This observation leads to our first general
result.

\begin{prop}
Let $G$ be a finite group and $p$ a prime. Then the conjugacy classes of
elements of $p$-power order form a clique in the NCC-graph.
\end{prop}

\begin{proof}
Let $P$ be a Sylow $p$-subgroup of $G$. By Sylow's theorem, every element of
$p$-power order is conjugate to an element of $P$; so $P$ meets every conjugacy
class of $p$-elements. Let $C_1$ and $C_2$ be two such classes, and take
$g_i\in C_i\cap P$. Then $\langle g_1,g_2\rangle\le P$, so this group is
a $p$-group, hence nilpotent.
\end{proof}

Thus the NCC-graph of $S_n$ has a clique whose size is the number of partitions
of $n$ into powers of $2$ (sequence A018819 in the OES).

The SCC-graph of $S_4$ is complete, because $S_4$ is a solvable group (and the
class of solvable groups is subgroup-closed). This is also a special case of the following
general result.

\begin{thm}
Let $G$ be a finite group. Then the CCC-graph (resp., the NCC-graph, the 
SCC-graph) of $G$ is complete if and only if $G$ is abelian (resp., nilpotent,
solvable).
\label{t:complete}
\end{thm}

The ``if'' statements are clear. For the converses, we use the following result.
Its roots lie in the work of Jordan.

\begin{prop}
Let $H$ be a proper subgroup of the finite group $G$. Then $G$ has a conjugacy
class disjoint from $H$.
\end{prop}

This holds because $G$ acts transitively on the set of right cosets of $H$ by
right multiplication; by Jordan's theorem, $G$ contains an element $x$ fixing no
coset, that is, no conjugate of $x$ lies in $H$.

\medskip

Thus, if we choose one element from each conjugacy class of $G$, these
elements generate $G$.

Now we prove Theorem~\ref{t:complete} for the CCC-graph. Suppose that the
CCC-graph is complete.  Choose any element $h$, say  $h\in C_1$. By our general
remark about conjugacy class graphs, there exist $h_i\in C_i$ for all $i$
such that $h_i$ commutes with $h=h_1$. Then $\langle h_1,\ldots,h_r\rangle=G$,
and so $h\in Z(G)$. Since $h$ was arbitrary, $G$ is abelian.

For the SCC-graph, this is immediate from the main theorem in~\cite{dghp},
according to which $G$ is solvable if and only if, for all $g,h\in G$, there
exists $x\in G$ such that $\langle g,h^x\rangle$ is solvable.

For the NCC-graph, we use \cite[Corollary E]{dghp}, which states that a finite
group $G$ is nilpotent if the following condition holds: for distinct primes
$p,q$, and for $g,h\in G$ where $g$ is a $p$-element and $h$ a $q$-element,
there exists $x\in G$ such that $g$ commutes with $h^x$. Now suppose that $G$
has complete NCC-graph, and let $p,q,g,h$ be as stated. Then there exists
$x\in G$ such that $\langle g,h^x\rangle$ is nilpotent. But in a nilpotent
group, a $p$-element and a $q$-element necessarily commute.

We would like to have a strengthening of this describing the dominant vertices
of one of our graphs (those joined to all others). The dominant vertices of
the commuting, nilpotent and solvable graphs are known; they are respectively
the centre, hypercentre, and solvable radical of the group
\cite[Theorem 11.2]{PC-21}. One might expect that the analogous result would
hold for conjugacy class graphs, since each of these sets is the union of
conjugacy classes. But this is not the case. The groups $\mathrm{PSL}(2,2^a)$
for $a\ge2$ have a single conjugacy class of involutions, and every element is
conjugate to its inverse by some involution. This means that for any element
$g$ there is an involution $h$ such that $\langle g,h\rangle$ is dihedral. So
the class of involutions is dominant in the SCC-graph, even though the
solvable radical is trivial. It is, however, true for the CCC-graph:

\begin{prop}\label{DV-CCC-graph}
The set of dominant vertices in the CCC-graph of a finite group $G$ is the
set of central conjugacy classes of $G$.
\end{prop}

\begin{proof} Clearly the central classes are dominant. Suppose that the
class of the element $g$ is dominant; that is, for all $h\in G$, there exists
$x\in G$ such that $g$ commutes with $h^x$. Then the centralizer of $g$ meets
every conjugacy class; by Jordan's result, it is the whole of $G$, so
$g\in Z(G)$ as required.
\end{proof}

\begin{prob} \label{Prob-DV-set}
 Describe the dominant vertices of the NCC- or SCC-graph of
a finite group.
\end{prob}

Finally in this section, we compare our conjugacy class graphs with the
conjugacy supergraphs as defined, for example, in~\cite{acns}. In these graphs,
the vertex set is the group $G$, and two vertices $g$ and $h$ are joined if and
only if there exist conjugates $g'$ and $h'$ of $g$ and $h$ respectively so that
$\langle g',h'\rangle$ has the appropriate property. These graphs are obtained
from the conjugacy class graphs defined here by ``inflating'' each vertex to the
number of vertices in its conjugacy class. In the other direction, we shrink a
conjugacy class to a single vertex.

It is clear that many properties of the two graphs (such as connectedness and
dominant vertices) will be unaltered by these transformations, while others
such as spectrum and clique number will change. We will not discuss this further.

The question of when these graphs are equal is answered by the next theorem.

\begin{thm}
Let $\mathcal{P}$ be one of the properties ``commutative'', ``nilpotent'',
or ``solvable''. A necessary and sufficient condition for the
$\mathcal{P}$ graph and the conjugacy super $\mathcal{P}$ graph to be equal
is as follows.
\begin{enumerate}
\item For $\mathcal{P}=\mbox{``commutative''}$: $G$ is a $2$-Engel group (one
satisfying the identity $[x,y,y]=1$).
\item For $\mathcal{P}=\mbox{``nilpotent''}$: $G$ is nilpotent.
\item For $\mathcal{P}=\mbox{``solvable''}$: $G$ is solvable.
\end{enumerate}
\label{t:super}
\end{thm}

\begin{proof} (a) This is \cite[Theorem 2]{acns}.

\smallskip

(b) It is clear that, if $G$ is nilpotent, then both graphs are complete. So
suppose that they are equal. Suppose that $G$ is not nilpotent. Then $G$
contains a Schmidt group (a minimal non-nilpotent group). These groups were
classified by Schmidt~\cite{schmidt}; a convenient reference is~\cite{ber}.

By inspection, any such group contains a $p$-element $x$ acting non-trivially
on a $q$-group $Q$, where $p$ and $q$ are distinct primes. If $y\in Q$ with
$y^x\ne y$, then
\[(x^{-1})^yx=[y,x]=y^{-1}y^x\]
is a non-identity $q$-element. But $\langle x,x\rangle$ is nilpotent, so by
assumption $\langle x^y,x\rangle$ is nilpotent. This is a contradiction since
all $p$-elements of a nilpotent group are contained in a single Sylow
$p$-subgroup.

\smallskip

(c) The key ingredient is the fact that any finite simple group can be
generated by two conjugate elements. In fact, by \cite{gk}, if G is a finite
simple group, then there exists $s\in G$ such that for all nontrivial $x\in G$
there exists $g\in G$ such that $\langle x,s^g\rangle=G$ (we can take $x=s$ to
get the previous claim).

It is clear that, if $G$ is solvable, then both graphs are complete. So
suppose that they are equal. Any two conjugates of an element $g$ are joined
in the $\mathcal{P}$ supergraph, and therefore are joined in the
$\mathcal{P}$ graph. Suppose that $G$ is not solvable. Let $N<M<G$ be a
subnormal series such that $M/N$ is a non-abelian simple group. As noted above,
there exists $g,y\in M$ such that $M/N=\langle Ng,Ng^y\rangle$. In particular,
$\langle Ng, Ng^y\rangle$ is non-solvable and hence $\langle g,g^y\rangle$ is
non-solvable. However, $\langle g,g\rangle=\langle g\rangle$ is solvable, which
is a contradiction. Therefore, $G$ is solvable, completing the proof.
\end{proof}

The simplicity of the conditions in (b) and (c) compared to (a) is striking.

\medskip

In the next three sections, we collate and survey some properties of the CCC-,
NCC- and SCC-graphs of finite groups.

\section{Commuting conjugacy class graph}

 We write $\mathcal{CCC}(G)$ to denote 
the CCC-graph of a group $G$. The  CCC-graph of $G$ is a graph whose vertex set is $\cl(G) := \{x^G : x \in G\}$, where $x^G$ denotes the conjugacy class of $x$ in $G$, and two distinct vertices $a^G$ and $b^G$ are adjacent if there exist some elements $x\in a^G$ and $y\in b^G$ such that $\langle x, y \rangle$ is an abelian group. In this section, we discuss results on various induced subgraphs of $\mathcal{CCC}(G)$. Herzog et al. \cite{HLM} considered three induced graphs of $\mathcal{CCC}(G)$ induced by 
$\cl(G\setminus 1)$, $\cl(G \setminus Z(G))$ and $\cl(G \setminus FC(G))$, where $\cl(S) = \{x^G : x \in S\}$ for any subset $S$ of $G$, and $FC(G) = \{x \in G : x^G \text{ is finite}\}$ is the \emph{FC-centre} of $G$.
 
\subsection{Connectivity of $\mathcal{CCC}(G)[\cl(G \setminus 1)]$}
 
Let $X$ be the class of groups which cannot be written as a union of conjugates of a proper subgroup. The following theorem gives a characterisation of residually $X$-group $G$ such that $\mathcal{CCC}(G)[\cl(G \setminus 1)]$ is complete.  
\begin{thm}\cite[Proposition 1]{HLM}
	Let $G$ be a residually $X$-group. Then the graph $\mathcal{CCC}(G)[\cl(G \setminus 1)]$ is complete if and only if $G$ is abelian. In particular, the claim holds if $G$ is residually (finite or solvable)-group.
\end{thm}

Connectedness of  $\mathcal{CCC}(G)[\cl(G \setminus 1)]$ is discussed in the next two theorems.
\begin{thm}\label{CCC-conn-1}
	\cite[Theorem 10 and 12]{HLM}
	Let $G$ be a finite solvable group or  a periodic solvable group. Then $\mathcal{CCC}(G)[\cl(G \setminus 1)]$ has at most two connected components, each of diameter $\leq 9$.
\end{thm}

\begin{thm} \label{CCC-conn-2}
	\cite[Theorems 13-14]{HLM}
	Let $G$ be a finite group or a locally finite group. Then $\mathcal{CCC}(G)[\cl(G \setminus 1)]$ has at most six connected components, each of diameter $\leq 19$.
\end{thm}

The following theorems give characterisation of  supersolvable/solvable groups such that   $\mathcal{CCC}(G)[\cl(G \setminus 1)]$ is disconnected.

\begin{thm}\cite[Proposition 7]{HLM}
	Let $G$ be a supersolvable group. Then the graph $\mathcal{CCC}(G)[\cl(G \setminus 1)]$ is disconnected if and only if $G$ is either of the groups given in the following two types:
	\begin{enumerate}
		\item $G = A \rtimes \langle x \rangle$, where $x \in G$, $|x| = 2$ and $A$ is a subgroup of $G$ on which $x$ acts fixed-point-freely.
		\item $G$ is finite and $G = A \rtimes B$, where $A$, $B$ are non-trivial subgroups of $G$, $A$ is nilpotent and $B$ is cyclic, and $B$ acts on $A$ fixed-point-freely (in particular, $G$ is a Frobenius group with kernel $A$ and a cyclic complement $B$). 
	\end{enumerate}
\end{thm}

\begin{thm}\label{CCC-disconnect}
	\cite[Theorem 16]{HLM}
	Let $G$ be a finite solvable group such that the graph $\mathcal{CCC}(G)[\cl(G \setminus 1)]$ is disconnected. Then there exists a nilpotent normal subgroup $H$ of $G$ such that one of the following holds:
	\begin{enumerate}
		\item $G = H \rtimes T$ is a Frobenius group with the kernel $H$ and a complement $T$.
		\item $G = (H \rtimes S) \rtimes \langle x \rangle$, where $S$ is a non-trivial cyclic subgroup of $G$ of odd order which acts fixed-point-freely on $H$, $x \in N_G(S)$ is such that $\langle x \rangle$ acts fixed-point-freely on $S$, and there exist $h_1 \in H \setminus \{ 1 \}$ and $i \in \mathbb{N}$ such that $x^i \neq 1$, which satisfy $[x^i, h_1] = 1$.
	\end{enumerate}
	Conversely, if either {\rm (a)} or {\rm (b)} holds, then $\mathcal{CCC}(G)[\cl(G \setminus 1)]$ is disconnected.
\end{thm}

\subsection{Properties of $\mathcal{CCC}(G)[\cl(G \setminus Z(G))]$}
In the following theorem  Herzog et al. \cite{HLM} determined all periodic groups $G$ such that $\mathcal{CCC}(G)[\cl(G \setminus Z(G))]$ is empty.

\begin{thm}\cite[Theorem 19]{HLM}
	Let $G$ be a periodic non-abelian group. Then $\mathcal{CCC}(G)[\cl(G \setminus Z(G))]$ is empty if and only if $G$ is isomorphic to $D_8$, $Q_8$ or $S_3$. 
\end{thm}

In 2016, Mohammadian et al. \cite{MEFW-2016} classified all finite groups $G$ such that the graph $\mathcal{CCC}(G)[\cl(G \setminus Z(G))]$ is triangle-free and obtained the following results. 

\begin{thm}\cite[Theorem 2.3]{MEFW-2016}
	If $G$ is a finite group of odd order and the graph $\mathcal{CCC}(G)[\cl(G \setminus Z(G))]$ is triangle-free, then $|G| = 21$  or $27$.
\end{thm}

\begin{thm}\cite[Theorem 3.4]{MEFW-2016}
	Suppose $G$ is a finite group of even order which is not a $2$-group and $\mathcal{CCC}(G)[\cl(G \setminus Z(G))]$ is triangle-free. If $|Z(G)| \neq 1$, then $G$ is isomorphic to $D_{12}$ or $T_{12} = \langle a, b : a^4 = b^3 =1, aba^{-1} = b^{-1}\rangle$.
\end{thm} 

\begin{thm}\cite[Theorem 3.5]{MEFW-2016}
	If $G$ is a centreless non-solvable finite group and $\mathcal{CCC}(G)[\cl(G \setminus Z(G))]$ is triangle-free, then $G$ is isomorphic to one of the groups $PSL(2, q)$ $(q\in \{4, 7, 9\})$, $PSL(3, 4)$ or $\SmallGroup(960, 11357)$. 
\end{thm}

\begin{thm}\cite[Theorem 3.6]{MEFW-2016}
	If $G$ is a centreless non-abelian solvable finite group and $\mathcal{CCC}(G)[\cl(G \setminus Z(G))]$ is triangle-free, then $G$ is isomorphic to one of the following groups: $S_3$, $D_{10}$, $A_{4}$, $S_4$, $\SmallGroup(72, 41)$, $\SmallGroup(192, 1023)$ or  $\SmallGroup(192, 1025)$.
\end{thm}

\begin{thm}\cite[Theorem 3.7]{MEFW-2016}
	If $G$ is a finite non-abelian $2$-group such that $\mathcal{CCC}(G)[\cl(G \setminus Z(G))]$ is triangle-free, then $\Phi(G) \leq Z(G)$ and $C_G(x) = \langle x, Z(G)\rangle$ whenever $x\in G\setminus Z(G)$. Furthermore, either $G \cong D_8$, $G \cong Q_8$ or $|G : Z(G)| = |Z(G)|$.
\end{thm}

\subsection{Structure of $\mathcal{CCC}(G)[\cl(G \setminus Z(G))]$}
In \cite{SA-2020,SA-CA-2020,Salah-2020} structures of commuting conjugacy class graphs of certain finite non-abelian groups were determined. In this section, we shall discuss the structures of CCC-graphs of dihedral group, generalized quaternion group, semi-dihedral group, the groups $U_{(n,m)}$, $V_{8n}$ and $G(p, m, n)$ along with some other groups such that $\frac{G}{Z(G)} \cong \mathbb{Z}_p\times \mathbb{Z}_p$ or $D_{2n}$, where $p$ is a prime and $D_{2n} = \langle x,y:~x^{n}=y^2=1,~yxy^{-1}=x^{-1} \rangle$.

\begin{thm}\cite[Theorem 1.2]{Salah-2020}\label{Struc-D2n}
	Let $G$ be a finite group with centre $Z(G)$ and $\frac{G}{Z(G)}$ is isomorphic to the dihedral group $D_{2n}$. 
	Then $$\mathcal{CCC}(G)[\cl(G \setminus Z(G))] = \begin{cases}
		K_{\frac{(n-1)|Z(G)|}{2}}\cup 2K_{\frac{|Z(G)|}{2}}, &\text{for }2\mid n\\
		K_{\frac{(n-1)|Z(G)|}{2}}\cup K_{|Z(G)|}, &\text{for }2\nmid n.
	\end{cases}$$
\end{thm}
As a corollary to Theorem \ref{Struc-D2n}, we get the structure of $\mathcal{CCC}(G)[\cl(G \setminus Z(G))]$ when $G$ is the dihedral group $D_{2n}$, the generalized quaternion group $Q_{4m}=\langle x,y:~x^{2m}=1,x^m=y^2,y^{-1}xy=x^{-1} \rangle$, the semi-dihedral group $SD_{8n}=\langle x,y:~x^{4n}=y^2=1,yxy=x^{2n-1}\rangle$ the group  $U_{(n,m)}=\langle x,y:~x^{2n}=y^m=1,~x^{-1}yx=y^{-1}\rangle$ and the group $U_{6n} =\langle x,y:~x^{2n}=y^3=1,~x^{-1}yx=y^{-1}\rangle$. Further, Salahshour and Ashrafi \cite[Proposition 2.4 and 2.6]{SA-CA-2020}  determined the structures of $\mathcal{CCC}(G)[\cl(G \setminus Z(G))]$ when $G$ is the group $V_{8n}=\langle x,y:$ $~x^{2n}=y^4=1,~yx=x^{-1}y^{-1},~y^{-1}x=x^{-1}y \rangle$ and the group $G(p, m, n) =$
$ \langle x,y:~x^{p^m}=y^{p^n}=[x, y]^p = 1, [x, [x, y]] = [y, [x, y]] = 1\rangle$ as given below:
\[
\mathcal{CCC}(V_{8n})[\cl(V_{8n} \setminus Z(V_{8n}))] = \begin{cases}
	K_{2n - 2} \cup 2K_2, &\text{for }2\mid n\\
	K_{2n - 1} \cup 2K_1, &\text{for }2\nmid n
\end{cases}
\]
and
\begin{align*}
\mathcal{CCC}(G(p, m, n))[\cl(G(p, m, n) &\setminus Z(G(p, m, n)))]\\
& = 2 K_{p^{m+n-1}-p^{m+n-2}} \cup (p^n - p^{n - 1}) K_{p^{m - n}(p^n - p^{n - 1})}.
\end{align*}

Note that all the groups considered above are AC-groups (non-abelian groups whose  centralizers of non-central elements are   abelian).
Salahshour and Ashrafi \cite{SA-CA-2020} also obtained the structure of $\mathcal{CCC}(G)[\cl(G \setminus Z(G))]$ if $G$ is a finite AC-group.  Let  $\cent(G) = \{ C_G(a) : a\in G \}$, where $C_G(a)$ is the centralizer of  $a \in G$. Consider the equivalence relation $\sim$ on $\cent(G) \setminus \{ G \}$ given by 
$C_G(a) \sim C_G(b)$ if and only if  $C_G(a)$ and $C_G(b)$ are conjugate in $G$. Then we have the following result.
\begin{thm}\cite[Theorem 3.3]{SA-CA-2020}\label{CCC-CA}
	Let $G$ be a finite AC-group with centre $Z(G)$. Then 
	\[
	\mathcal{CCC}(G)[\cl(G \setminus Z(G))] = \bigcup\limits_{\frac{C_G(a)}{\sim}\in EC(G)} K_{n_{\frac{C_G(a)}{\sim}}}
	\]
	where $EC(G) = \frac{Cent(G)\setminus \{ G \}}{\sim}$ is the set of all equivalence classes of $\sim$ and $n_{\frac{C_G(a)}{\sim}} = \frac{|C_G(a)| - |Z(G)|}{[N_G(C_G(a)) : C_G(a)]}$.
\end{thm} 
Salahshour and Ashrafi \cite{SA-2020}  determined the structures of $\mathcal{CCC}(G)[\cl(G \setminus Z(G))]$ when $G$ is a finite non-abelian group such that $\frac{G}{Z(G)}$ has order $p^2$ or $p^3$ as given in the following theorems. 
\begin{thm}\cite[Theorem 3.1]{SA-2020}\label{CCC-ZpZp}
	Let $G$ be a finite non-abelian group with centre $Z(G)$ and $\frac{G}{Z(G)} \cong \mathbb{Z}_p \times \mathbb{Z}_p$, where $p$ is prime. Then  $\mathcal{CCC}(G)[\cl(G \setminus Z(G))] = (p + 1)K_n$, where $n = \frac{(p - 1)|Z(G)|}{p}$.  
\end{thm} 
\begin{thm}\cite[Theorem 3.3]{SA-2020}
	Let $G$ be a finite non-abelian group with centre $Z(G)$ and $|\frac{G}{Z(G)}|  = p^3$, where $p$ is a prime. Then one of the following is satisfied:	
	\begin{enumerate}
		\item If $\frac{G}{Z(G)}$ is abelian then $\mathcal{CCC}(G)[\cl(G \setminus Z(G))] = K_m \cup p^2 K_n$ or $(p^2 + p + 1) K_n$, where $m = \frac{(p^2 - 1)|Z(G)|}{p}$ and $n = \frac{(p - 1)|Z(G)|}{p^2}$.
		\item If $\frac{G}{Z(G)}$ is non-abelian then $\mathcal{CCC}(G)[\cl(G \setminus Z(G))] = K_m \cup kp K_{n_1} \cup (p - k)K_{n_2}$,  $(kp + 1)K_{n_1} \cup (p + 1 - k)K_{n_2}$, $K_m \cup pK_{n_2}$, $(p^2 + p + 1)K_{n_1}$ or $K_{n_1} \cup (p + 1) K_{n_2}$, where $m = \frac{(p^2 - 1) |Z(G)|}{p}$, $n_1 = \frac{(p - 1)|Z(G)|}{p^2}$, $n_2 = \frac{(p - 1)|Z(G)|}{p}$, $1 \leq k \leq p$. 
	\end{enumerate}
\end{thm}

As a corollary, it follows that $\mathcal{CCC}(G)[\cl(G \setminus Z(G))]= (p+1)K_{p(p-1)}$ or $K_{(p^2-1)}\cup pK_{p-1}$ if $G$ is a non-abelian $p$-group of order $p^4$.
Ashrafi and Salahshour \cite[Theorem 1.2]{AS-2023} also obtained the structure of $\mathcal{CCC}(G)[\cl(G \setminus Z(G))]$ when  $\frac{G}{Z(G)}$ is isomorphic to $\mathbb{Z}_{p^2}\rtimes \mathbb{Z}_{p^2}$, where $p$ is a prime. In a recent work, Rezaei and Foruzanfar \cite{RF-2024} have  determined the structure of $\mathcal{CCC}(G)[\cl(G \setminus Z(G))]$ when 
$\frac{G}{Z(G)}$ is isomorphic to a Frobenius group of order $pq$ or $p^2q$, where $p, q$ are  primes.

\subsection{Genus of $\mathcal{CCC}(G)[\cl(G \setminus Z(G))]$}\label{CCC-genus}
For any graph $\Gamma$, we write $\gamma(\Gamma)$ to denote its genus. The genus of $\Gamma$ is the  smallest integer $k \geq 0$ such that $\Gamma$ can be embedded on the surface obtained by attaching $k$ handles to a sphere.
If $\gamma(\Gamma)$ is equal to $0, 1, 2$, or $3$, then $\Gamma$ is called planar, toroidal, double-toroidal, or triple-toroidal, respectively. Clearly,  $\gamma(K_1) = \gamma(K_2) =0$. For $n \geq 3$,  by \cite[Theorem 6-38]{AT-1973}, we have
\[
\gamma(K_n) = \bigg\lceil \frac{(n - 3)(n - 4)}{12} \bigg\rceil,
\]
where $\lceil a \rceil$ denotes the smallest integer greater than or equal to $a$ for any real number $a$. It is worth mentioning that finite non-abelian groups $G$ for which $\mathcal{C}(G)[G \setminus Z(G)]$ (the induced subgraph of commuting graph of $G$ induced by $G \setminus Z(G)$) is planar have been characterised (see \cite[Theorem 2.2]{AFK2015}), toroidal (see \cite[Theorem 2.2]{AFK2015} and \cite[Theorem 3.3]{DN1}), double-toroidal (see \cite[Theorem 3.3]{N-2024}) and triple-toroidal (see \cite[Theorem 3.7]{N-2024}). In this regard, we have the following problem.
\begin{prob}\label{Prob-genus-01}
Characterize all finite non-abelian groups $G$ such that the induced subgraph  $\mathcal{CCC}(G)[\cl(G \setminus Z(G))]$ of $\mathcal{CCC}(G)$  is planar, toroidal, double-toroidal or triple-toroidal.
\end{prob}
This problem was considered by Bhowal and Nath \cite{BN-2021} and they characterised the dihedral groups, generalized quaternion groups and semidihedral groups such that $\mathcal{CCC}(G)[\cl(G \setminus Z(G))]$ is planar, toroidal, double-toroidal or triple-toroidal. We have the following theorems for instance. 
\begin{thm}\cite[Theorem 2.2]{BN-2022}\label{Genus-D_{2n}}
	Let $G$ be the dihedral group $D_{2n}$. Then
	\begin{enumerate}
		\item $\mathcal{CCC}(G)[\cl(G \setminus Z(G))]$ is planar if and only if $3\leq n \leq 10$.
		\item $\mathcal{CCC}(G)[\cl(G \setminus Z(G))]$ is toroidal if and only if $11\leq n \leq 16$.
		\item $\mathcal{CCC}(G)[\cl(G \setminus Z(G))]$ is double-toroidal if and only if $n = 17,18$.
		\item $\mathcal{CCC}(G)[\cl(G \setminus Z(G))]$ is triple-toroidal if and only if $n = 19,20$.
		\item $\gamma(\mathcal{CCC}(G)[\cl(G \setminus Z(G))])=\begin{cases}
			\left\lceil\frac{(n-7)(n-9)}{48}\right\rceil, & \text{ for $2\nmid n$ and $n\geq 21$ } 
			\vspace{.2cm}\\
			\left\lceil\frac{(n - 8)(n - 10)}{48}\right\rceil, & \text{ for $2\mid n$ and $n\geq 22$. }
		\end{cases}$
	\end{enumerate}
\end{thm}
\begin{thm}\cite[Theorem 2.4]{BN-2022}
	Let $G$ be the generalized quaternion group $Q_{4m}$. Then
	\begin{enumerate}
		\item $\mathcal{CCC}(G)[\cl(G \setminus Z(G))]$ is planar if and only if $m= 2,3,4$ or $5$.
		\item $\mathcal{CCC}(G)[\cl(G \setminus Z(G))]$ is toroidal if and only if $m= 6,7$ or $8$.
		\item $\mathcal{CCC}(G)[\cl(G \setminus Z(G))]$ is double-toroidal if and only if $m=9$.
		\item $\mathcal{CCC}(G)[\cl(G \setminus Z(G))]$ is triple-toroidal if and only if $m=10$.
		\item $\gamma(\mathcal{CCC}(G)[\cl(G \setminus Z(G))]) =  \left\lceil\frac{(m - 4)(m - 5)}{12}\right\rceil$ for $m\geq 11$.
	\end{enumerate}
\end{thm}
\begin{thm} \cite[Theorem 2.3]{BN-2022}
	Let $G$ be the semidihedral group $SD_{8n}$. Then
	\begin{enumerate}
		\item $\mathcal{CCC}(G)[\cl(G \setminus Z(G))]$ is planar if and only if $n= 2$ or $3$.
		\item $\mathcal{CCC}(G)[\cl(G \setminus Z(G))]$ is toroidal if and only if $n=4$.
		\item $\mathcal{CCC}(G)[\cl(G \setminus Z(G))]$ is double-toroidal if and only if $n=5$.
		\item $\mathcal{CCC}(G)[\cl(G \setminus Z(G))]$ is not triple-toroidal.
		\item $\gamma(\mathcal{CCC}(G)[\cl(G \setminus Z(G))])=\begin{cases}
			\left\lceil\frac{(n - 3)(2n - 5)}{6}\right\rceil, & \text{ for $2\nmid n$ and $n\geq 7$ }
			\vspace{.2cm}\\
			\left\lceil\frac{(n - 2)(2n - 5)}{6}\right\rceil, & \text{ for $2\mid n$ and $n\geq 6$. }
		\end{cases}$
	\end{enumerate}
\end{thm}
Bhowal and Nath \cite{BN-2022} also considered the groups $V_{8n}$, $U_{(n, m)}$ and $G(p, m, n)$ in their study and obtained the following result.  
\begin{thm}\cite[Corollary 2.8]{BN-2022}
	Let $G$ be a group isomorphic to $D_{2n}$, $SD_{8n}$, $Q_{4m}$, $V_{8n}$, $U_{(n, m)}$ or $G(p, m, n)$. Then
	\begin{enumerate}
		\item $\mathcal{CCC}(G)[\cl(G \setminus Z(G))]$ is planar if and only if $G = D_{6}, D_{8}, D_{10}, D_{12}, D_{14}$, $D_{16}$, $D_{18}, D_{20},$  $SD_{16}$, $SD_{24}$, $Q_{8}, Q_{12}, Q_{16}, Q_{20}, V_{16}, U_{(2, 2)}$, $U_{(2, 3)}$, $U_{(2, 4)}$, $U_{(2, 5)}$, $U_{(2, 6)}$, $U_{(3, 2)}$, $U_{(3, 3)}$, $U_{(3, 4)}$, $U_{(4, 2)}$, $U_{(4, 3)}$, $U_{(4, 4)}$,  $G(2, 1, 1)$, $G(3, 1, 1)$, $G(5, 1, 1)$, $G(2, 2, 1)$, $G(2, 3, 1)$, $G(2, 1, 2)$, $G(2, 2, 2)$ or $G(2, 1, 3)$.
		
		\item $\mathcal{CCC}(G)[\cl(G \setminus Z(G))]$ is toroidal if and only if $G = D_{22}$, $D_{24}$, $D_{26}$, $D_{28}$, $D_{30}$, $D_{32}$, $SD_{32}$, $Q_{24}$, $Q_{28}$, $Q_{32}$, $V_{24}$, $V_{32}$, $U_{(2, 7)}$, $U_{(2, 8)}$,  $U_{(3, 5)}$ or $U_{(3, 6)}$.
		
		\item $\mathcal{CCC}(G)[\cl(G \setminus Z(G))]$ is double-toroidal if and only if $G = D_{34}$, $D_{36}$, $SD_{40}$, $Q_{36}$, $U_{(2, 9)}$, $U_{(2, 10)}$, $U_{(4, 5)}$, $U_{(4, 6)}$, $U_{(5, 2)}$, $U_{(5, 3)}$,  $U_{(6, 2)}$, $U_{(6, 3)}$, $U_{(7, 2)}$, $U_{(7, 3)}$ or $G(3, 1, 2)$.
		\item $\mathcal{CCC}(G)[\cl(G \setminus Z(G))]$ is triple-toroidal if and only if $G = D_{38}$, $D_{40}$, $Q_{40}$, $V_{40}$, $U_{(3, 7)}$, $U_{(3, 8)}$, $U_{(5, 4)}$, $U_{(6, 4)}$ or $U_{(7, 4)}$.
	\end{enumerate}
\end{thm}
It may be interesting to continue similar study for the groups with known/unknown structures of $\mathcal{CCC}(G)[\cl(G \setminus Z(G))]$ and answer Problem \ref{Prob-genus-01}.

\subsection{Various spectra and energies of $\mathcal{CCC}(G)[\cl(G \setminus Z(G))]$}
The spectrum of a finite graph $\Gamma$ with vertex set $V(\Gamma)$, denoted by $\spec(\Gamma)$, is the set of eigenvalues of its adjacency matrix with multiplicities. If $\spec(\Gamma) = \{\alpha_1^{a_1}, \alpha_2^{a_2},\ldots, \alpha_k^{a_k}\}$ for some $\Gamma$ then we mean that $\alpha_1, \alpha_2,\ldots,\alpha_k$ are the eigenvalues of the adjacency matrix of $\Gamma$ with multiplicities $a_1, a_2, \ldots, a_k$ respectively. Similarly,   $\L-spec(\Gamma)$ and $\Q-spec(\Gamma)$  denote the Laplacian spectrum (L-spectrum) and signless Laplacian spectrum (Q-spectrum) of $\Gamma$ i.e., the  set of eigenvalues of  the Laplacian and signless Laplacian matrices of $\Gamma$ respectively. A graph $\Gamma$ is called integral/L-integral/Q-integral if $\spec(\Gamma)$/$\L-spec(\Gamma)$/$\Q-spec(\Gamma)$ contains only integers. To determine all the  integral/ L-integral/Q-integral graphs is a general problem in graph theory. Various spectra of $\mathcal{C}(G)[G \setminus Z(G)]$  were computed in \cite{DN1,DN2,DN3,Nath-2018} and obtained various groups such that $\mathcal{C}(G)[G \setminus Z(G)]$ is integral/L-integral/Q-integral. Note that the following problem is still open.
\begin{prob}\label{Prob-integral}
Determine all the finite non-abelian groups $G$ such that $\mathcal{C}(G)[G \setminus Z(G)]$ is integral/L-integral/Q-integral.
\end{prob}
 
The energy, Laplacian energy (L-energy) and signless Laplacian energy (Q-energy) of $\Gamma$ denoted by $E(\Gamma)$, $LE(\Gamma)$ and $LE^+(\Gamma)$ respectively are given by 
\[
E(\Gamma) := \sum\limits_{\alpha \in \spec(\Gamma)}  |\alpha| ,\qquad LE(\Gamma) := \sum\limits_{\beta \in \L-spec(\Gamma)} \bigg|\beta-\frac{tr(\mathcal{D}(\Gamma))}{|V(\Gamma)|}\bigg|
\]  
\[
\text{ and } LE^+(\Gamma) := \sum\limits_{\gamma \in \Q-spec(\Gamma)} \bigg|\gamma - \frac{tr(\mathcal{D}(\Gamma))}{|V(\Gamma)|} \bigg|,
\] 
where $\mathcal{D}(\Gamma)$ is the degree matrix of $\Gamma$ and  $tr(\mathcal{D}(\Gamma))$ is  the trace of $\mathcal{D}(\Gamma)$. In 1978, Gutman \cite{Gutman-1978} introduced the notion of energy of a graph. In Huckel theory,  $\pi$-electron energy of a conjugated carbon molecule is  approximated by $E(\mathcal{G})$. Subsequently, Gutman and Zhou \cite{GZ-2006} in 2006 and Abreua et al. \cite{ACGMR11} in  2008 introduced L-energy and Q-energy of a graph.  Applications of these energies can be found in  crystallography, theory of macromolecules,  analysis and
comparison of protein sequences,  network analysis, satellite communication,  image analysis and processing etc. (see \cite{GF-2019}   and the references therein).
In 2009, Gutman et al. \cite{GAVBR} conjectured (E-LE conjecture) that
\begin{equation}\label{conjecture 1}
	E(\Gamma) \leq LE(\Gamma).
\end{equation}
Though \eqref{conjecture 1} was disproved in \cite{LL, SSM} people wanted to know whether this conjecture is true for various graphs defined on groups. The following problem for $\mathcal{C}(G)[G \setminus Z(G)]$  is  considered in \cite{DBN-KJM-2020,DN-IJPAM-2021}. 
\begin{prob}\label{Prob-E-LE}
	Determine all the finite non-abelian groups $G$ such that $\mathcal{C}(G)[G \setminus Z(G)]$ satisfy the following inequalities:
\begin{enumerate}
		\item $E(\mathcal{C}(G)[G \setminus Z(G)]) \leq LE(\mathcal{C}(G)[G \setminus Z(G)])$.
		\item $LE^+(\mathcal{C}(G)[G \setminus Z(G)]) \leq LE(\mathcal{C}(G)[G \setminus Z(G)])$.
	\end{enumerate}
\end{prob}
A graph $\Gamma$ is called hyperenergetic, \, L-hyperenergetic  and \, Q-hyperenergetic  if $E(\Gamma)$ $ > E(K_{|V(\Gamma)|})$, $LE(\Gamma) > LE(K_{|V(\Gamma)|})$ and $LE^+(\Gamma) > LE^+(K_{|V(\Gamma)|})$ respectively. The concept of hyperenergetic graph was given by Walikar et al. \cite{Walikar} and Gutman \cite{Gutman1999}, independently in 1999. The concept of L-hyperenergetic and Q-hyperenergetic graph can be found in \cite{FSN-2020}. Again, $\Gamma$ is called borderenergetic (introduced by Gong et al. \cite{GLXGF-2015}), L-borderenergetic (introduced by Tura \cite{Tura-2017}) and Q-borderenergetic (introduced by Tao et al. \cite{TH-2018}) if 
$E(\Gamma) = E(K_{|V(\Gamma)|})$, $LE(\Gamma) = LE(K_{|V(\Gamma)|})$ and $LE^+(\Gamma) = LE^+(K_{|V(\mathcal{G})|})$ respectively.
The following conjecture was posed by Gutman \cite{Gutman-1978} in 1978.  
\begin{conj}\label{con-hyper}
	Any finite graph $\Gamma \ncong K_{|v(\mathcal{G})|}$ is non-hyperenergetic. 
\end{conj} 
This conjecture  was also disproved by different mathematicians providing counter examples (see \cite{Gutman-2011}).
However, the search  for counter examples   to  Conjecture \ref{con-hyper} continued. In \cite{SND-PJM-2022}, the following  problem for $\mathcal{C}(G)[G \setminus Z(G)]$  was  considered.
\begin{prob}\label{Prob-hyper}
Determine all the finite non-abelian groups $G$ such that $\mathcal{C}(G)[G \setminus Z(G)]$\\ is  hyperenergetic/ \, borderenergetic/ \, L-hyperenergetic/ \, L-borderenergetic/ Q-hyper-\\energetic/ Q-borderenergetic.	
\end{prob}

In \cite{BN-2021,BN-2020}, Bhowal and Nath considered  problems corresponding to the Problems \ref{Prob-integral}--\ref{Prob-hyper} for $\mathcal{CCC}(G)[\cl(G \setminus Z(G))]$. They considered the dihedral groups, generalized quaternion groups and semidihedral groups and obtained the following results.
\begin{thm}\cite[Theorem 3.1]{BN-2021}\label{CCC(G)-D-2n}
	If $G$ is the dihedral group $D_{2n}$, then
	\begin{enumerate}
		\item $\spec(\mathcal{CCC}(G)[\cl(G \setminus Z(G))])$\\
		
		\qquad\qquad\qquad $  = \begin{cases}\left\{(0)^{1}, (-1)^{\frac{n - 3}{2}}, \left(\frac{n - 3}{2}\right)^{1}\right\}, &\text{for }2\nmid n\\
			\left\{(0)^{2}, (-1)^{\frac{n}{2} - 2}, \left(\frac{n}{2} - 2\right)^{1}\right\}, & \text{ for }2\mid n \text{ and } 2\mid \frac{n}{2}\\
			\left\{(-1)^{\frac{n}{2} - 1}, (1)^{1},  \left(\frac{n}{2} - 2\right)^{1}\right\}, & \text{ for }2\mid n \text{ and } 2\nmid\frac{n}{2} 
		\end{cases}$ 
		
		\vspace{.2cm}
		
		and $E(\mathcal{CCC}(G)[\cl(G \setminus Z(G))])= \begin{cases} n - 3, & \text{ for }2\nmid n\\
			n - 4, & \text{ for }2\mid n \text{ and } 2\mid \frac{n}{2}\\
			n - 2, & \text{ for }2\mid n \text{ and } 2\nmid\frac{n}{2}. 
		\end{cases}$
		
		\item $\L-spec(\mathcal{CCC}(G)[\cl(G \setminus Z(G))])$\\
		
		\qquad\qquad\qquad$ = \begin{cases}\left\{(0)^{2}, \left(\frac{n - 1}{2}\right)^{\frac{n - 3}{2}}\right\}, & \text{ for }2\nmid n\\
			\left\{(0)^{3}, \left(\frac{n}{2}- 1\right)^{\frac{n}{2} - 2}\right\}, & \text{ for }2\mid n \text{ and } 2\mid \frac{n}{2}\\
			\left\{(0)^{2}, 2^1, \left(\frac{n}{2}- 1\right)^{\frac{n}{2} - 2}\right\}, & \text{ for }2\mid n \text{ and } 2\nmid\frac{n}{2} 
		\end{cases}$ 
		
		\vspace{.2cm}
		
\noindent and $LE(\mathcal{CCC}(G)[\cl(G \setminus Z(G))])\!=\! \begin{cases} 
			\frac{2(n - 1)(n - 3)}{n + 1}, & \text{ for }2\nmid n\\
			\frac{3(n - 2)(n - 4)}{n + 2}, & \text{ for }2\mid n \text{ and } 2\mid \frac{n}{2}\\
			4, & \text{ for $n = 6$}\\
			\frac{(n - 4)(3n - 10)}{n + 2}, & \text{ for }2\mid n,~ n \geq 10\\
			&\qquad \text{ and } 2\nmid \frac{n}{2}. 
		\end{cases}$
		\item  $\Q-spec(\mathcal{CCC}(G)[\cl(G \setminus Z(G))])$\\
		
		\qquad\qquad$ = \begin{cases}\left\{(0)^{1}, (n - 3)^{1}, \left(\frac{n - 5}{2}\right)^{\frac{n - 3}{2}}\right\}, &  \text{ for }2\nmid n\\
			\left\{(0)^{2}, (n - 4)^{1}, \left(\frac{n}{2}- 3\right)^{\frac{n}{2} - 2}\right\}, &  \text{ for }2\mid n \text{ and } 2\mid \frac{n}{2}\\
			\left\{(0)^{1}, (2)^1, (n - 4)^{1}, \left(\frac{n}{2}- 3\right)^{\frac{n}{2} - 2}\right\}, & \text{ for }2\mid n \text{ and }2\nmid \frac{n}{2}
		\end{cases}$ 
		
		\vspace{.2cm}
		
\noindent \!\!\!\!\!\! and $LE^+(\mathcal{CCC}(G)[\cl(G \setminus Z(G))])= \begin{cases} 
			\frac{(n - 3)(n + 3)}{n + 1}, & \text{ for }2\nmid n\\
			\frac{(n - 4)(n + 6)}{n + 2}, & \text{ for } n = 4, 8\\
			\frac{2(n - 2)(n - 4)}{n + 2}, & \text{ for }2\mid n, \frac{n}{2} 
			\text{ and } n \geq 12\\
			4, & \text{ for $n = 6$}\\
			\frac{22}{3}, & \text{ for $n = 10$}\\
			\frac{2(n - 2)(n - 6)}{n + 2}, & \text{ for }2\mid n,~ 2 \nmid \frac{n}{2}\\
			& \qquad  \text{ and } n \geq 14.
		\end{cases}$
	\end{enumerate}
\end{thm}

\begin{thm}\cite[Theorem 3.2]{BN-2021}\label{Q(4m)}
	If $G$ is the generalized quaternion group $Q_{4m}$, then
	\begin{enumerate}
		\item $\spec(\mathcal{CCC}(G)[\cl(G \setminus Z(G))]) = \begin{cases}
			\left\{(-1)^{m - 1}, (1)^{1},  (m - 2)^{1}\right\}, &\text{ for } 2\nmid m\\
			\left\{(-1)^{m - 2}, (0)^{2},  (m - 2)^{1}\right\}, &\text{ for } 2\mid m 
		\end{cases}$ 
		
		and $E(\mathcal{CCC}(G)[\cl(G \setminus Z(G))])= \begin{cases} 
			2m - 2, &\text{ for } 2\nmid m\\
			2m - 4, &\text{ for } 2\mid m. 
		\end{cases}$
		
		\item $\L-spec(\mathcal{CCC}(G)[\cl(G \setminus Z(G))]) = \begin{cases}
			\left\{(0)^{2}, (2)^1,  (m - 1)^{m - 2}\right\}, &\text{ for } 2\nmid m\\
			\left\{(0)^{3},   (m - 1)^{m - 2}\right\}, &\text{ for } 2\mid m 
		\end{cases}$ 
		
\noindent \!\!\!\!\!\!		and $LE(\mathcal{CCC}(G)[\cl(G \setminus Z(G))])= \begin{cases} 
			4, & \text{ for $m = 3$}\\
			\frac{2(m - 2)(3m - 5)}{m + 1}, & \text{ for } 2\nmid m \text{ and } m \geq 5\\
			\frac{6(m - 1)(m - 2)}{m + 1}, & \text{ for } 2\mid m.
		\end{cases}$
		
		\item $\Q-spec(\mathcal{CCC}(G)[\cl(G \setminus Z(G))])$
		
		$\qquad\qquad = \begin{cases}
			\left\{(2)^1, (0)^{1}, (2m - 4)^{1}, (m - 3)^{m - 2}\right\}, &\text{for } 2\nmid m\\
			\left\{(0)^{2}, (2m - 4)^{1}, (m - 3)^{m - 2}\right\}, &\text{for } 2\mid m 
		\end{cases}$ 
		
\noindent \!\!\!\!\!\!		and $LE^+(\mathcal{CCC}(G)[\cl(G \setminus Z(G))])= \begin{cases} 
			4, & \text{ for $m = 3$}\\
			\frac{22}{3}, & \text{ for $m = 5$}\\
			\frac{4(m - 1)(m - 3)}{m + 1}, & \text{ for } 2\nmid m \text{ and } m \geq 7\\
			\frac{2(m - 2)(m + 3)}{m + 1}, & \text{ for $m = 2, 4$} \\
			\frac{4(m - 1)(m - 2)}{m + 1}, & \text{ for } 2\mid m \text{ and } m \geq 6.
		\end{cases}$
	\end{enumerate}
\end{thm}

\begin{thm}\cite[Theorem 3.5]{BN-2021}\label{SD(8n)}
	If $G$ is the semidihedral group $SD_{8n}$, then
	\begin{enumerate}
		\item $\spec(\mathcal{CCC}(G)[\cl(G \setminus Z(G))]) = \begin{cases} 
			\left\{(-1)^{2n}, (3)^{1},  \left(2n - 3\right)^{1}\right\}, & \text{ for $2\nmid n$}\\
			\left\{(-1)^{2n - 2}, (0)^{2},  \left(2n - 2\right)^{1}\right\}, & \text{ for $2\mid n$} 
		\end{cases}$ 
		
		and $E(\mathcal{CCC}(G)[\cl(G \setminus Z(G))])= \begin{cases} 
			4n, & \text{ for $2\nmid n$}\\
			4n - 4, & \text{ for $2\mid n$}. 
		\end{cases}$
		
		\item $\L-spec(\mathcal{CCC}(G)[\cl(G \setminus Z(G))]) = \begin{cases} 
			\left\{(0)^{2}, (4)^3,  \left(2n - 2\right)^{2n - 3}\right\}, & \text{ for $2\nmid n$}\\
			\left\{(0)^{3},   \left(2n - 1\right)^{2n - 2}\right\}, & \text{ for $2\mid n$} 
		\end{cases}$ 
		
		and $LE(\mathcal{CCC}(G)[\cl(G \setminus Z(G))])= \begin{cases} 
			12, & \text{ for $n = 3$}\\
			\frac{2(2n - 3)(5n - 11)}{n + 1}, & \text{ for $2\nmid n$ and $n \geq 5$}\\
			\frac{6(2n - 1)(2n - 2)}{2n + 1}, & \text{ for $2\mid n$.} 
		\end{cases}$
		
		\item $\Q-spec(\mathcal{CCC}(G)[\cl(G \setminus Z(G))])$
		
		$\qquad\qquad = \begin{cases} 
			\left\{(6)^1, (2)^{3}, (4n - 6)^{1}, \left(2n - 4\right)^{2n - 3}\right\}, & \text{ for $2\nmid n$}\\
			\left\{(0)^{2}, (4n - 4)^{1}, \left(2n - 3\right)^{2n - 2}\right\}, & \text{ for $2\mid n$} 
		\end{cases}$ 
		
		and $LE^+(\mathcal{CCC}(G)[\cl(G \setminus Z(G))])= \begin{cases} 
			12, & \text{ for $n = 3$}\\
			22, & \text{ for $n = 5$}\\
			\frac{16(n - 1)(n - 3)}{n + 1}, & \text{ for $2\nmid n$  and $n \geq 7$}\\
			\frac{28}{5}, & \text{ for $n = 2$}\\
			\frac{4(2n - 1)(2n - 2)}{2n + 1}, & \text{ for $2\mid n$ and $n \geq 4$.} 
		\end{cases}$ 
	\end{enumerate}
\end{thm}

Bhowal and Nath \cite{BN-2021,BN-2020} also considered the groups $V_{8n}$, $U_{(n, m)}$ and $G(p, m, n)$ in their study and obtained the following  results.
\begin{thm}\cite[Corollary 3.6]{BN-2021}
If $G$ is isomorphic to $D_{2n}$, $Q_{4m}$, $SD_{8n}$, $V_{8n}$, $U_{(n, m)}$ and $G(p, m, n)$ then $\mathcal{CCC}(G)[\cl(G \setminus Z(G))]$ is  integral, L-integral and Q-integral.
\end{thm}
Comparing various  energies of $\mathcal{CCC}(G)[\cl(G \setminus Z(G))]$ they also obtained the following results.
\begin{thm}\cite[Theorem 4.6]{BN-2021}\label{summarized}
	Let $G$ be a finite non-abelian group and $\Gamma = \mathcal{CCC}(G)[\cl(G \setminus Z(G))]$. Then
	\begin{enumerate}
		\item  $E(\Gamma) = LE^+(\Gamma) = LE(\Gamma  )$ if $G$ is isomorphic to $D_{6}, D_{8}, D_{12}, Q_{8}, Q_{12}$, $U_{(n, 2)}$, $U_{(n, 3)}, U_{(n, 4)}$ $(n\geq 2)$, $V_{16}$, $SD_{24}$ or $G(p, m, 1)$ $(p \geq 2, m \geq 1)$.
		
		\item  $LE^+(\Gamma) < E(\Gamma) < LE(\Gamma)$ if $G$ is isomorphic to $D_{20},  Q_{20}$, $ U_{(2, 5)}, U_{(3, 5)}$, $U_{(2, 6)}$ or $G(2, 2, 2)$.
		
		\item  $E(\Gamma) < LE^+(\Gamma) < LE(\Gamma)$ if $G$ is isomorphic to $D_{14}, D_{16}, D_{18}$,  $D_{2n}$ $(n \geq 11)$, $ Q_{16}, Q_{24}, Q_{4m}$ $(m\geq 8)$, $U_{(n, 5)}$  $(n \geq 4)$, $U_{(n, m)}$ $(m \geq 6  \text{ and } n \geq 3)$,  $U_{(n, m)}$ $(m \geq 8 \text{ and } n \geq 2)$, $V_{8n}$ $(n \geq 3)$, $SD_{16}$,  $SD_{8n}$ $(n \geq 4)$, $G(2, m, 2)$ $(m \geq 3)$, $G(p, m, 2)$ $(p \geq 3, m \geq 1)$ or $G(p, m, n)$ $(n \geq 3, p \geq 2, m \geq 1)$.
		
		\item   $E(\Gamma) = LE^+(\Gamma) < LE(\Gamma)$ if $G$ is isomorphic to $Q_{28}$ or $U_{(2, 7)}$
		
		\item  $E(\Gamma) < LE^+(\Gamma)$ $ = LE(\Gamma)$ if $G$ is isomorphic to $D_{10}$ and $G(2, 1, 2)$.
	\end{enumerate}
\end{thm}

\begin{thm}\cite[Theorem 5.6]{BN-2021}
	Let $G$ be a finite non-abelian group. Then
	\begin{enumerate}
		
		\item $\mathcal{CCC}(G)[\cl(G \setminus Z(G))]$ is neither hyperenergetic, borderenergetic, \\ L-hyperenergetic,  L- borderenergetic,  Q-hyperenergetic nor Q-borderenergetic if $G$ is isomorphic to $D_{8}$, $D_{12}$, $D_{2n}\,(\text{$n$ is odd})$, $Q_{8}$, $Q_{12}$,\, $Q_{16}$,\, $U_{(2,6)}$,\,  $U_{(n,3)}$, \,$U_{(n,4)}$\,$(n\geq 2)$,\, $V_{16}$,\, $SD_{16}$,\, $SD_{24}$,\, $G(p, m, 1)\,\, (p\geq 2$ and $m\geq 1)$, $G(2, 1, 2)$ or $G(2, 2, 2)$.
		
		\item $\mathcal{CCC}(G)[\cl(G \setminus Z(G))]$ is L-borderenergetic but neither hyperenergetic, borderenergetic, L-borderenergetic,  Q-hyperenergetic nor Q-borderenergetic if $G$ is isomorphic to $Q_{20}$ or $U_{(2,5)}$.
		
		\item $\mathcal{CCC}(G)[\cl(G \setminus Z(G))]$ is L-hyperenergetic but neither hyperenergetic, borderenergetic, L- borderenergetic, \, Q-hyperenergetic nor\, Q-borderenergetic if $G$ is isomorphic to \,$D_{16}$, \,$D_{20}$, \,$D_{24}$, \,$D_{28}$, $Q_{24}$, $Q_{28}$, $U_{(3,5)}$, $U_{(3,6)}$, $U_{(2,7)}$, $V_{24}$, $V_{32}$, $G(2, 3, 2)$ or $G(2, 1, 3)$.
		
		\item $\mathcal{CCC}(G)[\cl(G \setminus Z(G))]$ is L-hyperenergetic and Q-borderenergetic but neither hyperenergetic, borderenergetic,  L-borderenergetic nor Q-hyperenergetic if $G$ is isomorphic to $SD_{40}$.

		\item[{\rm(e)}] $\mathcal{CCC}(G)[\cl(G \setminus Z(G))]$ is L-hyperenergetic and Q-hyperenergetic but neither hyperenergetic,   borderenergetic, \,L-borderenergetic nor \,Q-borderenergetic if $G$ is isomorphic to \,$D_{2n}$ \,($n$ is even, $n\geq 16$), $Q_{4m}\,\, (m\geq 8)$, $U_{(n,5)}\,\, (n\geq 4)$, $U_{(n,6)}\,\, (n\geq 4)$,\, $U_{(n,7)}\,\, (n\geq 3)$, $U_{(n,m)}$\,\, $(n\geq 2$  and $m\geq 8)$, \,$V_{8n}$ \,\,$(n\geq 5)$, \,$SD_{32}$, \,$SD_{8n}\,\,(n\geq 6)$, \,$G(2, m, 2)\,\,(m\geq 4)$,\, $G(p, m, 2)$\,\, $(p\geq 3$ and $m\geq 1)$, $G(2, m, 3)\,\,(m\geq 2)$ or $G(p, m, n)\,\,(n \geq 4, p\geq 2 \text{ and } m\geq 1)$. 
	\end{enumerate}
\end{thm}

\begin{thm}\cite[Theorem 5.7]{BN-2021}
	Let $G$ be a finite non-abelian group.
	\begin{enumerate}
		
		\item[{\rm(a)}] Then $\mathcal{CCC}(G)[\cl(G \setminus Z(G))]$ is neither hyperenergetic nor borderenergetic, for $G$ is isomorphic to $D_{2n}$, $Q_{4m}$, $U_{(n, m)}$, $V_{8n}$, $SD_{8n}$ or $G(p, m, n)$.
		
		\item[{\rm(b)}] Then $\mathcal{CCC}(G)[\cl(G \setminus Z(G))]$ is L-hyperenergetic, for $G$ is isomorphic to $D_{2n}$ $(n$ is even, $n\geq 8)$,  $Q_{4m}$ $(m\geq 6)$, \,\,$U_{(n,5)}$ $(n\geq 3)$, \,\,$U_{(n,6)}$ $(n\geq 3)$,\,\, $U_{(n,m)}$ $(n\geq 2 \text{ and } m\geq 7)$,\,\, $V_{8n}$ $(n\geq 3)$, \,\,$SD_{8n}$ $(n\geq 4)$,\,\, $G(2, m, 2)$ $(m\geq 3)$, $G(p, m, 2)$ $(p\geq 3 \text{ and } m\geq 1)$, $G(2, m, 3)$ $(m\geq 1)$ or $G(p, m, n)$ $(n \geq 4, p\geq 2$ and $m\geq 1)$.
		
		\item[{\rm(c)}] Then $\mathcal{CCC}(G)[\cl(G \setminus Z(G))]$ is L-borderenergetic, for $G$ is isomorphic to $Q_{20}$ or $U_{(2,5)}$.

		\item[{\rm(d)}] Then $\mathcal{CCC}(G)[\cl(G \setminus Z(G))]$ is Q-hyperenergetic, for $G$ is isomorphic to \,\,$D_{2n}$\,\, $n$ is even, $n\geq 16)$,\, $Q_{4m}$ $(m\geq 8)$, \,$U_{(n,5)}$ $(n\geq 4)$, \,$U_{(n,6)}$\, $(n\geq 4)$, $U_{(n,7)} $ $ (n\geq 3)$,\,\, $U_{(n,m)}$ $(n\geq 2 \text{ and } m\geq 8)$, \,\,\,$V_{8n}\,\,(n\geq 5)$,\, $SD_{32}$,\,$SD_{8n}$ $(n\geq 6)$, $G(2, m, 2)$ $(m\geq 4)$, $G(p, m, 2)$ $(p\geq 3 \text{ and } m\geq 1)$, $G(2, m, 3)$ $(m\geq 2)$ or $G(p, m, n)$ $(n \geq 4$, $p\geq 2$ and $m\geq 1)$.

		\item[{\rm(e)}] Then $\mathcal{CCC}(G)[\cl(G \setminus Z(G))]$ is Q-borderenergetic, for $G$ is isomorphic to  $SD_{40}$. 
	\end{enumerate}
\end{thm}

We conclude this section noting that problems analogous to Problems \ref{Prob-integral} -- \ref{Prob-hyper} for various common neighborhood spectrum and energies of $\mathcal{CCC}(G)[\cl(G \setminus Z(G))]$ were considered in   \cite{JN-2025,JN-2024}. Common neighborhood  spectrum/energy, common neighborhood Laplacian spectrum/energy and common neighborhood signless Laplacian spectrum/energy of graphs were introduced in \cite{ASG} and \cite{NJD-2023}. Various common neighborhood spectrum and energies of $\mathcal{C}(G)[G \setminus Z(G)]$ were considered in \cite{FSN-2020,FSN-2021,NFDS-2021}.

\subsection{Properties of $\mathcal{CCC}(G)[\cl(G \setminus FC(G))]$}
\label{s:fc}

The FC-centre of a group $G$ is the set of elements $x \in G$ such that $x^G$ is finite.  Herzog et al. \cite{HLM} obtained the following results for the induced subgraph $\mathcal{CCC}(G)[\cl(G \setminus FC(G))]$ of $\mathcal{CCC}(G)$ induced by $\cl(G)\setminus \{g^G : g \in FC(G)\}$ when $G$ is a periodic group.

\begin{thm}\cite[Theorem 22]{HLM}\label{FC theorem}
	Let $G$ be a periodic group such that the graph $\mathcal{CCC}(G)[\cl(G \setminus FC(G))]$ is empty. Write $F=FC(G)$ and 
	suppose that there exists $xF \in G/ F$ such that $|xF|=3$. Then $G$ has the following structure: 
	$$ G = F \rtimes \langle a, b \rangle,$$
	where $|a| = 3$, $|b| = 2$, $a^b = a^{-1}$ 
	$(\text{i.e.} \langle a, b \rangle \cong S_3)$, $F$ is an elementary abelian $2$-group, $F = D_1\times D_2$, where $D_1 = \{ [d, b] : d\in F \}$, $D_2 = \{ [d, ab] : d\in F \}$, $D_1$ and $D_2$ are infinite subgroups of $F$ and $a$ acts fixed-point-freely on $F$.
	
	 Conversely, if $G$ has the above structure, then \, $F =  FC(G)$ \, and the graph $\mathcal{CCC}(G)[\cl(G \setminus FC(G))]$ is empty.
\end{thm}

\begin{thm}\cite[Theorem 23]{HLM}
	Let $G$ be a periodic group such that the graph $\mathcal{CCC}(G)[\cl(G \setminus FC(G))]$ is empty. Write $F=FC(G)$. Then $G$ is locally finite and either $G$ is a hypercentral $2$-group with $G/ F$ abelian of exponent $2$ or $G/ F$ is finite. In the latter case, either $G/F$ is a finite elementary abelian $2$-group or $G/F \cong S_3$ and $G$ has the structure described in Theorem \ref{FC theorem}.
\end{thm}

\section{Nilpotent conjugacy class graph}
We write $\mathcal{NCC}(G)$ to denote the
NCC-graph of a group $G$. The  NCC-graph of $G$ is a graph whose vertex set is $\cl(G)$ and two distinct vertices $a^G$ and $b^G$ are adjacent if there exist some elements $x\in a^G$ and $y\in b^G$ such that $\langle x, y \rangle$ is a nilpotent group. In this section, we discuss various results on  induced subgraphs of $\mathcal{NCC}(G)$. Mohammadian and Erfanian \cite{ME-2017} considered two induced subgraphs of $\mathcal{NCC}(G)$ induced by  $\cl(G\setminus 1)$ and $\cl(G \setminus \nil(G))$,
where $\nil(G) := \{g \in G : \langle g, x \rangle \text{ is nilpotent for all } x \in G\}$ is the hypercentre of $G$.

\subsection{Connectivity of $\mathcal{NCC}(G)[\cl(G \setminus 1)]$}

Mohammadian and Erfanian \cite{ME-2017} obtained the following results analogous to Theorem \ref{CCC-conn-1} and Theorem \ref{CCC-conn-2}.

\begin{thm}\label{NCC-conn}
Let $G$ be any group.
\begin{enumerate}
\item \cite[Theorem 2.6-2.7]{ME-2017} If $G$ is finite solvable or  periodic solvable then the graph    $\mathcal{NCC}(G)[\cl(G \setminus 1)]$ has at most two connected components whose diameters are at most $7$.
\item \cite[Theorem 2.10-2.11]{ME-2017} If $G$ is finite or locally finite then $\mathcal{NCC}(G)[\cl(G \setminus 1)]$ has at most six connected components whose diameters are at most $10$.
\end{enumerate}	
\end{thm}

Mohammadian and Erfanian \cite{ME-2017} also obtained the following characterisations of supersolvable and solvable groups $G$  such that $\mathcal{NCC}(G)[\cl(G \setminus 1)]$ is disconnected.

\begin{thm}\cite[Theorem 3.2]{ME-2017}
	Let $G$ be a supersolvable group. Then the graph $\mathcal{NCC}(G)[\cl(G \setminus 1)]$ is disconnected if and only if one of the following holds:
	\begin{enumerate}
		\item If $G$ is an infinite group, then $G=H \rtimes \langle a \rangle$, where $a \in G$, $|a| = 2$ and $H$ is a subgroup of $G$ on which $a$ acts fixed-point-freely.
		\item If $G$ is a finite group, then $G=H\rtimes K$ is a Frobenius group with kernel $H$ and a cyclic complement $K$.
	\end{enumerate}
\end{thm}

\begin{thm} \label{NCC-disconnect-2}
	\cite[Theorem 3.4]{ME-2017}
	Let $G$ be a finite solvable group with disconnected graph $\mathcal{NCC}(G)[\cl(G \setminus 1)]$. Then there exists a nilpotent normal subgroup $N$ of $G$ such that one of the following holds:
	\begin{enumerate}
		\item $G=N \rtimes H$ is a Frobenius group with the kernel $N$ and a complement $H$.
		\item $G = (N \rtimes L)\rtimes \langle a \rangle$, where $L$ is a non-trivial cyclic subgroup of $G$ of odd order which acts fixed-point-freely on $N$, $a\in N_G(L)$ is such that $\langle a \rangle$ acts fixed-point-freely on $L$, and there exist $a\in N\setminus \{1\}$ and $i \in \mathbb{N}$ such that $x^i \neq 1$ and $[a, x^i] = 1$.
	\end{enumerate}
	Conversely, if either {\rm (a)} or {\rm (b)} holds, then $\mathcal{NCC}(G)[\cl(G \setminus 1)]$ is disconnected.
\end{thm}
Notice that the two situations in Theorem \ref{CCC-disconnect} and Theorem \ref{NCC-disconnect-2}, where we get disconnected  $\mathcal{CCC}(G)[\cl(G \setminus 1)]$ and $\mathcal{NCC}(G)[\cl(G \setminus 1)]$, are  identical. Therefore, the following problem arises naturally.
\begin{prob}\label{CCC=NCC}
Determine whether $\mathcal{CCC}(G)[\cl(G \setminus 1)] = \mathcal{NCC}(G)[\cl(G \setminus 1)]$ if and only if one of the cases in Theorem \ref{NCC-disconnect-2} holds.
\end{prob}

\subsection{Properties of $\mathcal{NCC}(G)[\cl(G \setminus \nil(G))]$}
 Mohammadian and Erfanian  \cite{ME-2017} also considered the subgraph \, $\mathcal{NCC}(G)[\cl(G \setminus \nil(G))]$ of $\mathcal{NCC}(G)$ induced by the set $\cl(G \setminus \nil(G))$ in their study. They obtain the following characterisations of finite non-nilpotent groups $G$ such that $\mathcal{NCC}(G)[\cl(G \setminus \nil(G))]$ is empty/triangle-free.  
\begin{thm}\cite[Theorem 4.3]{ME-2017}
	Let $G$ be a finite non-nilpotent group. Then the graph $\mathcal{NCC}(G)[\cl(G \setminus \nil(G))]$ is an empty graph if and only if $G \cong S_3$.
\end{thm}
\begin{thm}
	Let $G$ be a finite non-nilpotent group.
\begin{enumerate}
\item \cite[Theorem 4.8]{ME-2017} If  $|G|$ is odd then  $\mathcal{NCC}(G)[\cl(G \setminus \nil(G))]$ is triangle-free if and only if $|G|=21$.
\item \cite[Theorem 4.9]{ME-2017} If  $|G|$ is even then  $\mathcal{NCC}(G)[\cl(G \setminus \nil(G))]$ is triangle-free if and only if $G$ is isomorphic to one of the groups $S_3$, $D_{10}$, $D_{12}$, $A_4$, $T_{12}$ or $\PSL(2, q)$ where $q\in \{4, 7, 9\}$. 
\end{enumerate}	
\end{thm}

Note that the structure of  graph $\mathcal{NCC}(G)[\cl(G \setminus \nil(G))]$ is not realized much. If the structures of $\mathcal{NCC}(G)[\cl(G \setminus \nil(G))]$ for various families of finite groups are known then one can consider problems similar to Problems \ref{Prob-integral} -- \ref{Prob-hyper} for the graph $\mathcal{NCC}(G)[\cl(G \setminus \nil(G))]$. Therefore,  the following problem is worth mentioning.
\begin{prob}\label{NCC-structure}
Determine the structure of $\mathcal{NCC}(G)[\cl(G \setminus \nil(G))]$ for various families of finite non-nilpotent groups.
\end{prob}
 We conclude this section with the following problem analogous to Problem \ref{Prob-genus-01}.
\begin{prob}\label{genus-NCC}
	Characterize all finite non-nilpotent groups $G$ such that the induced subgraph  $\mathcal{NCC}(G)[\cl(G \setminus \nil(G))]$ of $\mathcal{NCC}(G)$  is planar, toroidal, double-toroidal or triple-toroidal.
\end{prob}
\section{Solvable cojugacy class graph}
We write $\mathcal{SCC}(G)$ to denote the 
SCC-graph of a group $G$. The  SCC-graph of $G$ is a graph whose vertex set is $\cl(G)$ and two distinct vertices $a^G$ and $b^G$ are adjacent if there exist some elements $x\in a^G$ and $y\in b^G$ such that $\langle x, y \rangle$ is a solvable group. In this section, we discuss results on an induced subgraph of $\mathcal{SCC}(G)$. Bhowal et al. \cite{BCNS-2023,BCNS-2024} considered the  subgraph of $\mathcal{SCC}(G)$ induced by  $\cl(G\setminus 1)$.

\subsection{Connectivity of $\mathcal{SCC}(G)[\cl(G \setminus 1)]$} Not much is known about the connectivity of $\mathcal{SCC}(G)[\cl(G \setminus 1)]$. We have the following problem whose answer is not known.
\begin{prob}\label{SCC-diam-prob}
If $G$ is a finite non-solvable group then determine the number of  components of $\mathcal{SCC}(G)[\cl(G \setminus 1)]$ and an upper bound for diameters of its components.
\end{prob}
The answers to Problem  \ref{SCC-diam-prob} for $\mathcal{CCC}(G)[\cl(G \setminus 1)]$ and $\mathcal{NCC}(G)[\cl(G \setminus 1)]$ are  given in Theorem \ref{CCC-conn-2} and Theorem \ref{NCC-conn} respectively. The following results are known regarding the connectivity of the graph $\mathcal{CCC}(G)[\cl(G \setminus 1)]$.
\begin{thm}\cite[Theorem 2.1]{BCNS-2023}
	Let $G$ be a finite group. Then $\mathcal{SCC}(G)[\cl(G \setminus 1)]$ is complete if and only if $G$ is solvable.
\end{thm} 
\begin{thm}\cite[Theorem 2.9]{BCNS-2023}
	If $G$, $H$ are arbitrary non-trivial groups then the graph $\mathcal{SCC}(G \times H)[\cl(G \times H \setminus (1_G, 1_H))]$ is connected with  diameter  $\leq 3$. In particular, $\mathcal{SCC}(G \times G)[\cl(G \times G \setminus (1_G, 1_G))]$ is connected with diameter $\leq 3$. Further, $\diam(\mathcal{SCC}(G \times G)[\cl(G \times G \setminus (1_G, 1_G))]) = 3$ if and only if $\diam(\mathcal{SCC}(G)[\cl(G \setminus 1)]) \geq 3$ (possibly infinite).
\end{thm}

\begin{thm}\cite[Theorem 4.4]{BCNS-2023}
	Let $G$ be a finite group. Let  $H$ and $K$ be two subgroups of $G$ such that $H$ is  normal in $G$,  $G = HK$ and $\mathcal{SCC}(H)[\cl(H \setminus 1)]$, $\mathcal{SCC}(K)[\cl(K \setminus 1)]$ are connected. If there exist two elements $h \in H \setminus \{1\}$ and $x \in G \setminus H$ such that $h^G$ and $x^G$ are connected in $\mathcal{SCC}(G)[\cl(G \setminus 1)]$ then $\mathcal{SCC}(G)[\cl(G \setminus 1)]$ is connected.
\end{thm}

\begin{thm}\cite[Theorem 3.3]{BCNS-2023}
	Let $G$ be a finite group. If $G$ has an element of order $n = \Pi_{i=1}^mp_i^{k_i}$, where $p_i$'s are distinct primes. Then $\mathcal{SCC}(G)[\cl(G \setminus 1)]$ has a clique of size $\Pi_{i=1}^m(k_i+1)-1$.
\end{thm}

\begin{thm}\label{SCC-cliqueNo}
	\cite[Theorem 3.5]{BCNS-2023}
	For any positive integer $d$, there are only finitely many finite groups $G$ such that the clique number of $\mathcal{SCC}(G)$ is $d$.
\end{thm}

\begin{prob}
Do the analogues of Theorem \ref{SCC-cliqueNo} hold for the graphs $\mathcal{CCC}(G)$ and $\mathcal{NCC}(G)$?
\end{prob}
 
We conclude this section with the following result which shows that the graph $\mathcal{SCC}(G)[\cl(G \setminus 1)]$ is triangle-free when $G \cong \{1\}, \mathbb{Z}_2, \mathbb{Z}_3$ or $S_3$, the
symmetric group of degree~$3$.
\begin{thm}\cite[Theorem 3.4]{BCNS-2023}
	With the exception of the cyclic groups of orders $1$, $2$ and $3$ and the
	symmetric group of degree~$3$, every finite group $G$ has the property that
	$\mathcal{SCC}(G)[\cl(G \setminus 1)]$ contains a triangle (that is, has girth~$3$).
\end{thm}

\subsection{Genus of $\mathcal{SCC}(G)[\cl(G \setminus 1)]$}
We have seen various results on   genus of the graph $\mathcal{CCC}(G)[\cl(G \setminus Z(G))]$ in Subsection \ref{CCC-genus}. The genus of $\mathcal{SCC}(G)[\cl(G \setminus \sol(G))]$, where $\sol(G) := \{g \in G : \langle g, x \rangle \text{ is solvable for all } x \in G\}$, the solvable radical of $G$, is not studied so far. However, the following problem is considered in \cite{BCNS-2024}. 
\begin{prob}\label{SCC-gen-01}
	Characterize all finite groups $G$ such that the graph  $\mathcal{SCC}(G)[\cl(G \setminus 1)]$  is planar, toroidal, double-toroidal,  triple-toroidal or projective.
\end{prob}
Let $N_k$ be the connected sum of $k$ projective planes. A simple graph $\Gamma$ which can be embedded in $N_k$ but not in $N_{k-1}$, is called a graph with crosscap $k$. We write $\bar{\gamma}(\Gamma)$ to denote the crosscap of $\Gamma$. A graph $\Gamma$ is called   projective  if $\bar{\gamma}(\Gamma) = 1$. Bhowal et al. \cite{BCNS-2024} obtained the following results related to Problem \ref{SCC-gen-01}.

\begin{thm}\cite[Theorem 3.1]{BCNS-2024} Let $G = D_{2n}$. Then 
	\begin{enumerate}
		\item $\mathcal{SCC}(G)[\cl(G \setminus 1)]$ is planar if and only if $n = 2, 3, 4, 5$ and $7$.
		\item $\mathcal{SCC}(G)[\cl(G \setminus 1)]$ is toroidal if and only if $n = 6, 8, 9, 10, 11$ and $13$.
		\item $\mathcal{SCC}(G)[\cl(G \setminus 1)]$ is double-toroidal if and only if $n = 12$ and $15$.
		\item $\mathcal{SCC}(G)[\cl(G \setminus 1)]$ is triple-toroidal if and only if $n = 14$ and $17$.
		\item $\gamma(\mathcal{SCC}(G)[\cl(G \setminus 1)])=\begin{cases}
			\left\lceil \frac{(n-5)(n-7)}{48}\right\rceil, & \text{when $n \geq 19$ and $n$ is odd}\vspace{.2cm}\\
			
			\left\lceil \frac{(n-2)(n-4)}{48}\right\rceil, & \text{when $n \geq 16$ and $n$ is even}. 
		\end{cases}$
		\item $\bar{\gamma}(\mathcal{SCC}(G)[\cl(G \setminus 1)]) = 0$ if and only if $n = 2, 3, 4, 5$ and $7$.
		
		\item $\mathcal{SCC}(G)[\cl(G \setminus 1)]$ is projective if and only if $n = 6, 8, 9$ and $11$.
		
		\item $\bar{\gamma}(\mathcal{SCC}(G)[\cl(G \setminus 1)])=\begin{cases}
			\left\lceil \frac{(n-5)(n-7)}{24}\right\rceil, & \text{when $n \geq 13$ and $n$ is odd}\vspace{.2cm}\\
			
			\left\lceil \frac{(n-2)(n-4)}{24}\right\rceil, & \text{when $n \geq 10$ and $n$ is even}. 
		\end{cases}$
	\end{enumerate}
\end{thm}

\begin{thm}\cite[Theorem 3.2]{BCNS-2024} Let $G = Q_{4m}$. Then
	\begin{enumerate}
		\item $\mathcal{SCC}(G)[\cl(G \setminus 1)]$ is planar if and only if $m = 1$ and $2$.
		\item $\mathcal{SCC}(G)[\cl(G \setminus 1)]$ is toroidal if and only if $m = 3, 4$ and $5$.
		\item $\mathcal{SCC}(G)[\cl(G \setminus 1)]$ is double-toroidal if and only if $m = 6$.
		\item $\mathcal{SCC}(G)[\cl(G \setminus 1)]$ is triple-toroidal if and only if $m = 7$.
		\item $\gamma(\mathcal{SCC}(G)[\cl(G \setminus 1)])= \left\lceil \frac{(m-1)(m-2)}{12}\right\rceil$ for $m \geq 8$.
		\item $\bar{\gamma}(\mathcal{SCC}(G)[\cl(G \setminus 1)]) = 0$ if and only if $m = 1$ and $2$.
		
		\item $\mathcal{SCC}(G)[\cl(G \setminus 1)]$ is projective if and only if $m = 3$ and $4$.
		
		\item $\bar{\gamma}(\mathcal{SCC}(G)[\cl(G \setminus 1)])=\begin{cases}
			3, & \text{when $m = 5$}\vspace{.2cm}\\
			
			\left\lceil \frac{(m-1)(m-2)}{6}\right\rceil, & \text{when $m \geq 6$}. 
		\end{cases}$
		
	\end{enumerate}
\end{thm}
\begin{thm}\cite[Theorem 3.3]{BCNS-2024}
	Let $G$ be a finite solvable group. Then
	\begin{enumerate}
		\item $\mathcal{SCC}(G)[\cl(G \setminus 1)]$ is planar if and only if $|\cl(G)| = 1, 2, 3, 4$ and $5$.
		\item $\mathcal{SCC}(G)[\cl(G \setminus 1)]$ is toroidal if and only if $|\cl(G)| = 6, 7$ and $8$.
		\item $\mathcal{SCC}(G)[\cl(G \setminus 1)]$ is double-toroidal if and only if $|\cl(G)| = 9$.
		\item $\mathcal{SCC}(G)[\cl(G \setminus 1)]$ is triple-toroidal if and only if $|\cl(G)| = 10$.
		\item $\gamma(\mathcal{SCC}(G)[\cl(G \setminus 1)])= \left\lceil \frac{(|\cl(G)|-4)(|\cl(G)|-5)}{12}\right\rceil$ for $|\cl(G)| \geq 11$.
		\item $\bar{\gamma}(\mathcal{SCC}(G)[\cl(G \setminus 1)]) = 0$ if and only if $|\cl(G)| = 1, 2, 3, 4$ and $5$.
		
		\item $\mathcal{SCC}(G)[\cl(G \setminus 1)]$ is projective if and only if $|\cl(G)| = 6$ and $7$.
		
		\item $\bar{\gamma}(\mathcal{SCC}(G)[\cl(G \setminus 1)])=\begin{cases}
			3, & \text{when $|\cl(G)| = 8$}\vspace{.2cm}\\
			
			\left\lceil \frac{(|\cl(G)|-4)(|\cl(G)|-5)}{6}\right\rceil, & \text{when $|\cl(G)| \geq 9$}. 
		\end{cases}$
		
	\end{enumerate}
	
\end{thm} 

\begin{thm}\cite[Theorem 3.4]{BCNS-2024} Let $G = S_n$. Then
	\begin{enumerate}
		\item
		$\mathcal{SCC}(G)[\cl(G \setminus 1)]$ is planar if and only if $n \leq 5$.
		\item
		If $n\geq 7$ then  $\mathcal{SCC}(G)[\cl(G \setminus 1)]$ is neither planar, toroidal, double-toroidal nor triple-toroidal.
		\item
		$\mathcal{SCC}(G)[\cl(G \setminus 1)]$ is not toroidal if $n = 6$.
		\item
		If $n \geq 6$ then $\mathcal{SCC}(G)[\cl(G \setminus 1)]$ is not projective.
	\end{enumerate}
\end{thm}

\begin{thm}\cite[Theorem 3.5]{BCNS-2024} Let $G = A_n$, the alternating group of degree $n$. Then
	\begin{enumerate}
		\item
		$\mathcal{SCC}(G)[\cl(G \setminus 1)]$ is  planar if and only if $n \leq 6$.
		\item
		If $n \geq 9$ then $\mathcal{SCC}(G)[\cl(G \setminus 1)]$ is  neither planar, toroidal, double-toroidal nor triple-toroidal.
		\item
		$\mathcal{SCC}(G)[\cl(G \setminus 1)]$ is  toroidal if and only if $n = 7$.
		\item
		If $n\geq 8$ then $\mathcal{SCC}(G)[\cl(G \setminus 1)]$ is not projective.
	\end{enumerate}
\end{thm}

Bhowal et al. \cite{BCNS-2024} also obtained the genus of $\mathcal{SCC}(G)[\cl(G \setminus 1)]$ when $G$ is certain projective special linear group. 
\begin{thm}\cite[Theorem 3.6]{BCNS-2024}
Let $G = \PSL(2,q)$, where $q=2^d$  with $d\ge3$. Then  
	$$\gamma (\mathcal{SCC}(G)[\cl(G \setminus 1)]) =\gamma(K_{q/2})+\gamma(K_{q/2+1}) \text{ or }
	\gamma(K_{q/2})+\gamma(K_{q/2+1})-1.$$
\end{thm}

Regarding bounds for genus and crosscap of $\mathcal{SCC}(G)[\cl(G \setminus 1)]$
we have the following results. 
\begin{thm}\cite[Theorem 2.4]{BCNS-2024}
	Given $n \geq 10$, let $k=p(\lfloor n/2\rfloor)-1$. If $G = S_n$, the symmetric group of degree $n$, and $\Gamma = \mathcal{SCC}(G)[\cl(G \setminus 1)]$ then
	\[
	\gamma(\Gamma) \ge \left\lceil\frac{(k-3)(k-4)}{12}\right\rceil \,\, \text{ and } \,\,
	\bar\gamma(\Gamma) \ge \left\lceil\frac{(k-3)(k-4)}{6}\right\rceil.
	\]
\end{thm}
\begin{thm}\label{SCC-gen-bd}
	\cite[Theorem 2.6]{BCNS-2024}
	Let $G$ be a finite non-solvable group with non-trivial centre $Z(G)$. Then 
	\[
	4\gamma(\mathcal{SCC}(G)[\cl(G \setminus 1)]) \geq (|Z(G)| - 3)(|\cl(G)| - |Z(G)| - 2).
	\]
\end{thm}

As an application of Theorem \ref{SCC-gen-bd}, we have the following bound for $\Pr(G)$ which is the probability that a randomly chosen pair of elements of $G$ commute; also known as  commutativity degree of $G$. 
\begin{thm}\label{Pr(G)-bound}
	\cite[Corollary 2.7]{BCNS-2024} Let $G$ be a finite non-solvable group  and $|Z(G)|> 3$. Then
	\[
	\Pr(G)  \leq \frac{4\gamma(\Gamma_{sc}(G)) + (|Z(G)| - 3)(|Z(G)| + 2)}{|G|(|Z(G)| - 3)}.
	\]
\end{thm}
Many other bounds for $\Pr(G)$ using various group theoretic notions can be found in \cite{GR-2006,nD-2010,DNP-2013}. It is worth noting that similar bounds for nilpotency degree \cite{DGW-1992} and solvability degree \cite{FGSV-2000}   can be obtained by obtaining bounds for 
 $\gamma(\mathcal{CCC}(G)[\cl(G \setminus 1)])$ and $\gamma(\mathcal{NCC}(G)[\cl(G \setminus 1)])$ 
similar to   Theorem \ref{SCC-gen-bd}.

It is intuitive that the order of a finite group $G$ is bounded (above) by a function of the genus of $\mathcal{SCC}(G)[\cl(G \setminus 1)]$. In this regard, we have the following problem. 
\begin{prob} \label{SCC-Prob-04}
	\cite[Problem 2.2]{BCNS-2024}
	Find an explicit bound for the order of a finite group $G$ for which
	\begin{enumerate}
		\item $\gamma(\mathcal{SCC}(G)[\cl(G \setminus 1)])=k$. 
		
		\item	$\bar\gamma(\mathcal{SCC}(G)[\cl(G \setminus 1)])=k$.
	\end{enumerate}
\end{prob}
We conclude this section noting that problems similar to Problem \ref{SCC-Prob-04} for the graphs $\mathcal{CCC}(G)[\cl(G \setminus 1)]$ and $\mathcal{NCC}(G)[\cl(G \setminus 1)]$    are worth considering.

\section{Concluding remarks}
In \cite{PC-21},  conditions for holding equalities in the hierarchy \eqref{H-1} were discussed for   various $\mathcal{P}$ graphs of finite groups. For instance, the cyclic graph of $G$ is equal to the commuting graph of $G$ if and only if $G$ contains no subgroup isomorphic to $\mathbb{Z}_p \times \mathbb{Z}_p$ (where $p$ is prime); 
the commuting graph of $G$ is equal to the nilpotent graph of $G$  if and only if the Sylow subgroups of $G$ are abelian; 
the nilpotent graph of $G$ is equal to the solvable graph of $G$  if and only if $G$ is nilpotent. It may be interesting to obtain conditions for holding equalities in the hierarchy \eqref{H-2} for various $\mathcal{PCC}$-graphs.

Let $DV(\mathcal{PCC}(G))$ be the set of all dominant vertices of $\mathcal{PCC}(G)$. In view of Proposition \ref{DV-CCC-graph}, $DV(\mathcal{CCC}(G)) = \cl(Z(G))$.
It is easy to see that 
$\cl(\nil(G)) \subseteq DV(\mathcal{NCC}(G))$ and $\cl(\sol(G))$ $\subseteq DV(\mathcal{SCC}(G))$. The following problem along with similar problems for  $\mathcal{CCC}(G)$ and $\mathcal{NCC}(G)$ are worth mentioning here.
\begin{prob}\label{size-DV}
\cite[Problem 3.8]{BCNS-2024}
Which non-solvable finite groups $G$ have the property that  	$|DV(\mathcal{SCC}(G))| = 2$? 
\end{prob} 	

 To make the question of connectedness of $\mathcal{P}$ conjugacy class graphs  interesting it is necessary to determine $DV(\mathcal{PCC}(G))$ (also see  Problem \ref{Prob-DV-set}). Now consider the following problem.
\begin{prob}\label{PCC-connect}
Determine whether the induced subgraphs   
$\mathcal{CCC}(G)[\cl(G\setminus Z(G))]$,
 $\mathcal{NCC}(G)[\cl(G)\setminus DV(\mathcal{NCC}(G))]$ and $\mathcal{SCC}(G)[\cl(G)\setminus DV(\mathcal{SCC}(G))]$ of $\mathcal{CCC}(G)$, $\mathcal{NCC}(G)$ and $\mathcal{SCC}(G)$ respectively are connected. Also, find upper bounds for the diameters of their components.
\end{prob}  
Note that the induced subgraphs  $\mathcal{NCC}(G)[\cl(G \setminus FC(G))]$ and $\mathcal{SCC}(G)[\cl(G \setminus FC(G))]$ of    $\mathcal{NCC}(G)$ and $\mathcal{SCC}(G)$ respectively are not studied yet. (Recall that $FC(G)$ is the
FC-centre of $G$, see Section~\ref{s:fc}.) Therefore, researchers working in this area  may consider these graphs in their study. 

The complements of cyclic/commuting/nilpotent/solvable graphs (in other words non-cyclic/non-commuting/non-nilpotent/non-solvable graphs) of finite groups were well-studied over the years (see \cite{AH-07,AH-09,AAM-06,DN-2018,FN-2020,AZ-10,Z-23,hr,BNN-22}). However, the complements of  $\mathcal{PCC}$-graphs of $G$ are not studied. Note that $DV(\mathcal{PCC}(G))$ is the set of isolated vertices of the complement of $\mathcal{PCC}(G)$. Thus, by Proposition \ref{DV-CCC-graph}, it follows that the set of isolated vertices of the complement of $\mathcal{CCC}(G)$ is $\cl(Z(G))$. The following problem is equivalent to Problem \ref{Prob-DV-set}.
\begin{prob} 
Describe the set of isolated  vertices of the complements of  NCC- or SCC-graph of
a finite group.
\end{prob}
Similarly, we have equivalent problems corresponding to Problems \ref{gp-determination}, \ref{CCC=NCC}, \ref{NCC-structure} and \ref{size-DV} for the complement of $\mathcal{PCC}(G)$.
 We conclude this paper noting that  problems analogous to  Problems \ref{Prob-genus-01}--\ref{Prob-hyper}, \ref{genus-NCC}, \ref{SCC-diam-prob}--\ref{SCC-Prob-04} and \ref{PCC-connect}   for   complements of  $\mathcal{PCC}$-graphs of $G$ are worth considering.

\section*{Acknowledgments}
F. E. Jannat would like to thank DST for the INSPIRE Fellowship (IF200226). 

The authors are grateful to Scott Harper for the proof of 
Theorem~\ref{t:super}(c).

\end{document}